\documentclass[10pt]{amsart}
\usepackage[letterpaper, margin=0.8in]{geometry}
\usepackage{graphicx}
\usepackage{amsmath,amssymb,amsthm,mathtools}
\usepackage[all]{xy}
\usepackage[utf8]{inputenc}
\usepackage[hidelinks]{hyperref}
\usepackage{enumitem}
\usepackage{xcolor}
\usepackage{tikz-cd}
\usepackage{tikz}
\usepackage{import}
\usepackage{marginnote,color}

\title{Torsors of the Jacobians of the universal Fermat curves}

\newtheorem{theorem}{Theorem}

\newtheorem{lemma}{Lemma}
\newtheorem{proposition}{Proposition}

\newtheorem{question}{Question}

\definecolor{yellow}{rgb}{0.99,0.99,0.80}
\begin{document}
%\pagecolor{yellow}

\author{Qixiao MA}
\address{Insitute of Mathematical Sciences, ShanghaiTech University, 201210, Shanghai, China (Current)}
\email{maqx1@shanghaitech.edu.cn}
\address{Institute of algebraic geometry, Leibniz University Hannover, Welfengarten 1, 30167 Hannover, Germany (Former)}
\email{qixiaoma@math.uni-hannover.de}

\begin{abstract}
Let $m\geq3$ be an integer. 
We show that every torsor of the Jacobian of the universal family of degree-$m$ Fermat curve is necessarily a connected component of the Picard scheme. We show that the Jacobian of the generic degree-$m$ Fermat curve has uncountably many non-isomorphic torsors. We give some results towards the Franchetta type problem for torsors of the Jacobian of the universal family of genus-$g$ curves over $\mathcal{M}_g$.

\end{abstract}
\maketitle
\setcounter{tocdepth}{1}

\tableofcontents
\section{Introduction}
\subsection{Franchetta type questions}
Let $g\geq3$ be an integer, so that the moduli stack $\mathcal{M}_g$ of smooth genus $g$ curves is generically an integral scheme. Let $C/K$ be the generic fiber of the universal family $\pi\colon \mathcal{M}_{g,1}\to \mathcal{M}_{g}$. The classical Franchetta theorem, originally posed as a conjecture in \cite{zbMATH03091363}, asserts that every line bundle on $C$ is necessarily a tensor power of the canonical line bundle: $$\mathrm{Pic}(C)=\mathbb{Z}\cdot[\omega_C],$$
see \cite{Schroeer-Strong-Franchetta} for the history of this problem and 
proofs of strengthened results in arbitrary characteristics. 
%The Franchetta theorem is a bit surprising because for a smooth complex curve $C_0$, the group of line bundles $\mathrm{Pic}(C_0)\cong\mathbb{C}^{g}/\mathbb{Z}^{2g}$ is uncountable, however if one tries to construct line bundles for each family of genus $g$ curves, that are compatible with respect to pullbacks, one could only expect canonical constructions.

While the classical Franchetta conjecture has been fully settled, the observation that ``there are only canonical constructions on generic objects'' motivated interesting problems in arithmetic geometry, most notably O'Grady's question whether zero cycles on generic polarized K3 surfaces are necessarily multiples of the Beauville-Voisin class \cite{zbMATH06448861}. For evidences of this conjecture, see \cite{zbMATH07004433}, \cite{zbMATH07495474}, \cite{zbMATH07690544} and \cite{zbMATH07577489}, etc.

\subsection{Motivation and summary}In this paper, instead of Chow group of zero cycles, we explore Franchetta type question for the Weil-Ch\^atelet group: the group of torsors of abelian schemes. We would like to determine if every torsor of the Jacobian of the universal genus $g$ curve is necessarily canonical, i.e., a connected components of the Picard scheme. 

While we do not solve this problem over $\mathcal{M}_g$, we fully solve this problem for the universal family of Fermat curves: In Theorem \ref{thm1}, we show that all torsors of the Jacobian of the universal family of smooth Fermat curves are canonical.
However, over the generic point, we show that there are uncountably many non-canonical torsors of the generic Fermat curves, see Theorem \ref{thm2}.
Then we use our methods to get
 partial results over $\mathcal{M}_g$.

\subsection{Relation with other works}
We point out relation of our work with some other recent works.
\begin{enumerate}[leftmargin=*]
\item In \cite{colliotthélène2023lowdegreeunramifiedcohomology}, Colliot-Th\'el\`ene
and Skorobogatov showed that the $i$-th unramified cohomology group of the generic diagonal hypersurface in the projective space of dimension $n\geq i +1$ is trivial for $i=1,2,3$.  
For the generic degree-$m$ Fermat curve $X/K$, Leray spectral sequence and Artin vanishing identifies $\mathrm{H}^1(K,\mathrm{Pic}^0_{X/K})$ with a finite extension of $\mathrm{Br}(X)/\mathrm{Br}(K)$. So our result on uncountability of this group illustrates that the vanishing result in \cite[Theorem 4.8]{colliotthélène2023lowdegreeunramifiedcohomology} is optimal for $i=2$. 
\item 
In \cite{Pirisi-Brauer-genus-3}, \cite{Pirisi-Brauer-Group-Universal}
the authors calculated Brauer group of moduli stack of curves, our Franchetta type question for torsors of Jacobian over $\mathcal{M}_g$ provides some possible applications of their work.
\end{enumerate}
\subsection{Acknowledgements}
The author thank Professor Colliot-Th\'el\`ene for helpful discussion on this topic when the author was a postdoc at Leibniz University Hannover, where the project began. The author thank Lian Duan and Yingdi Qin for important remarks which he overlooked. The author thank Shizhang Li, Shuai Wang and Dingxin Zhang for helpful conversations on calculating of cohomology of monodromy.

\section{Set ups and main results}
Throughout we work over $\mathbb{C}$. Let $\mathbb{P}^2_{[a_0:a_1:a_2]}\times\mathbb{P}^2_{[X_0:X_1:X_2]}$ be the product projective space, with bi-homogeneous coordinates $([a_0:a_1:a_2],[X_0:X_1:X_2])$. Let $m\geq3$ be an integer.

\subsection{The universal family of degree-\texorpdfstring{$m$}{m} Fermat curves} We call the hypersurface $$\mathcal{X}:=V(a_0X_0^m+a_1X_1^m+a_2X_2^m)\subset\mathbb{P}^2_{[a_0:a_1:a_2]}\times\mathbb{P}^2_{[X_0:X_1:X_2]}$$ the universal family of degree-$m$ Fermat curves. Let $\mathrm{pr}_1,\mathrm{pr}_2$ be the projections onto the factors:  $$\mathrm{pr}_1:\mathcal{X}\to \mathbb{P}^2_{[a_0:a_1:a_2]}, \ \ \mathrm{pr}_2:\mathcal{X}\to \mathbb{P}^2_{[X_0:X_1:X_2]}.$$ Note that: 
\begin{enumerate}
\item The fibers of $\mathrm{pr}_1\colon\mathcal{X}\to\mathbb{P}^2_{[a_0:a_1:a_2]}$ are plane curves of degree $m$.
\item The total space $\mathcal{X}$ is a  $\mathbb{P}^1$-bundle over $\mathbb{P}^2$ via $\mathrm{pr}_2$, more precisely, $\mathcal{X}\cong\mathbb{P}(\mathrm{pr}_{2,*}\mathrm{pr}_1^*\mathcal{O}_{\mathbb{P}^2}(1)).$
\end{enumerate}
Let $U=\mathbb{P}^2_{[a_0:a_1:a_2]}-V(a_0a_1a_2)$ be the complement of coordinate axes.
Let us call the base change $(\mathrm{pr}_1)_U\colon \mathcal{X}_U\to U$ the universal family of smooth degree-$m$ Fermat curves.

\subsection{The Picard scheme and its canonical torsors}
Let $\mathrm{Pic}^d_{\mathcal{X}_U/U}$ be the degree-$d$ component of the relative Picard scheme. The connected component $\mathrm{Pic}^0_{\mathcal{X}_U/U}$ is an abelian scheme over $U$ and is canonically isomorphic to its dual abelian scheme, therefore we also call $\mathrm{Pic}^0_{\mathcal{X}_U/U}$ the Jacobian of $\mathcal{X}_U/U$.

The components $\{\mathrm{Pic}_{\mathcal{X}_U/U}^d\}_{d\in\mathbb{Z}}$ are \'etale torsors under $\mathrm{Pic}^0_{\mathcal{X}_U/U}$, we call them the canonical torsors of $\mathrm{Pic}^0_{\mathcal{X}_U/U}$. The group of all \'etale torsors of $\mathrm{Pic}^0_{\mathcal{X}_U/U}$ are classified by $\mathrm{H}^1_{\textrm{et}}(U,\mathrm{Pic}^0_{\mathcal{X}_U/U})$. It contains the cyclic subgroup  generated by the class of $\mathrm{Pic}^1_{\mathcal{X}_U/U}$. 
Since $\mathrm{Pic}^d_{\mathcal{X}_U/U}$ can be identified with $\mathrm{Pic}^{d+m}_{\mathcal{X}_U/U}$ by tensoring with $\mathrm{pr}_2^*\mathcal{O}(1)$, we have $m[\mathrm{Pic}^1_{\mathcal{X}_U/U}]=(m+d)[\mathrm{Pic}^1_{\mathcal{X}_U/U}]$ and therefore the cyclic group $\langle[\mathrm{Pic}^1_{\mathcal{X}_U/U}]\rangle$ is torsion with order divisible by $m$.

%We will focus on \'etale torsors because the group scheme $\mathrm{Pic}^0_{\mathcal{X}_U/U}\to U$ is smooth, and all flat torsors of a  smooth group scheme are \'etale locally trivial, hence are torsors over the  \'etale site over $U$. Finally note that flat torsors of abelian schemes are represented by schemes, so there will be no ambiguity between flat $\mathrm{Pic}^0_{\mathcal{X}_U/U}$-torsors and elements in sheaf cohomology $\mathrm{H}^1_{\mathrm{et}}(U,\mathrm{Pic}^0_{\mathcal{X}_U/U})$.  

%Families of line bundles of relative degree $d$ on $\mathcal{X}_U/U$ gives rise to sections of $\mathrm{Pic}^d_{\mathcal{X}_U/U}$, however not every section of $\mathrm{Pic}_{\mathcal{X}_U/U}$ is represented by families line bundles, indeed, they are represented by families of twisted line bundles.

\subsection{Posing Franchetta type questions}

%The classical 
%Franchetta theorem 
%that when $g\geq3$, every line bundles on the generic genus $g$ curve is generated by the 

%on the generic genus-$g$ curves, one naturally wonders 
%if every torsor of the Jacobian of generic genus $g$ curve is necessarily canonical? 

When we pose the Franchetta question for line bundles, it doesn't matter if we consider the universal family or its generic fiber. However when we pose the question for torsors, there is subtlety. We explain the difference:

Let $\pi\colon\mathcal{C}\to U$ be a family of smooth curves over a smooth base, and let $C/K$ be the generic fiber
\begin{enumerate}
\item Any line bundle on the generic fiber of $\pi$ necessarily extends to a line bundle on $\mathcal{C}$ by taking closure of a corresponding Weil divisor. Closure of linearly equivalent Weil divisors may not be linearly equivalent on $\mathcal{C}$, the difference is measured by pullbacks of line bundles on $U$, which is in general fairly understood:$$\xymatrix{\mathrm{Pic}(U)\ar[r]^-{\pi^*}&\mathrm{Pic}(\mathcal{C})\ar[r]&\mathrm{Pic}(C)\ar[r]&0}.$$
%The difference between $\mathrm{Pic}(\mathcal{C})$ and $\mathrm{Pic}(C)$ is controlled by $\mathrm{Pic}(U)$
\item However, let $T/K$ be a torsor of $\mathrm{Pic}^0_{C/K}$, there is no apparent way to extend it to a torsor $\mathcal{T}/U$ of $\mathrm{Pic}^0_{\mathcal{C}/U}$. 
\end{enumerate}
Therefore we need to specify if we pose the Franchetta type question for the universal family or for the generic fiber of the universal family. This paper will show that the two set ups give quite different results.

\subsection{A Franchetta type result over \texorpdfstring{$U$}{U}}
We show that all torsors of the Jacobian of the universal family of smooth degree-$m$ Fermat curves are  necessarily isomorphic to $\mathrm{Pic}^d_{\mathcal{X}_U/U}$ %-- some connected component of the Picard scheme. 
We also determine the period of the generator $[\mathrm{Pic}^1_{\mathcal{X}_U/U}]\in\mathrm{H}^1_{\mathrm{et}}(U,\mathrm{Pic}^0_{\mathcal{X}_U/U})$.
\begin{theorem}\label{thm1}
The group of $\mathrm{Pic}_{\mathcal{X}_U/U}^0$-torsors over $U$ is a cyclic group generated by the class of $\mathrm{Pic}^1_{\mathcal{X}_U/U}$. The class $[\mathrm{Pic}^1_{\mathcal{X}_U/U}]$ has period $m$ when $m$ is odd and period $m/2$ when $m$ is even: 
$$\mathrm{H}^1_{\mathrm{et}}(U,\mathrm{Pic}^0_{\mathcal{X}_U/U})=\langle[\mathrm{Pic}^0_{\mathcal{X}_U/U}]\rangle\cong 
    \left\{ 
        \begin{array}{lc} 
            {\mathbb{Z}}/{m\mathbb{Z}},& 2\nmid m\\
            {\mathbb{Z}}/{\left(\frac{1}{2}m\right)\mathbb{Z}}, & 2\mid m
        \end{array}
    \right.$$ 
\end{theorem}
Running Leray spectral sequence, the problem is reduced to calculation of the Brauer group of the universal family, and the set of regular sections $\mathrm{Pic}_{\mathcal{X}_U/U}(U)$ of the relative Picard scheme. We calculate the Brauer group via standard ramification sequence, which is doable because the singular stratum of the universal family of Fermat curves is rather clear. The calculation of Brauer group reduces the rest of the problem to determining of sections of $\mathrm{Pic}^0_{\mathcal{X}_U/U}[m]$. With the help of \cite{Galois-action-homology-MR3596577}, we further reduce it to a linear algebra calculation.

It is fun to note that for even degree Fermat curves, there are rational points on $\mathrm{Pic}^{m/2}_{\mathcal{X}_U/U}$, which can be thought of as ``theta characteristics'' of $\mathrm{pr}_2^*\mathcal{O}(1)$.

\subsection{Non-canonical torsors over the generic point}
We show that the analogous Franchetta type result no longer holds over the generic point of $U$. 

Let $K=\mathrm{f.f.}(U)$ be the field of fractions of $U$, and let $X/K$ be the generic fiber of $\mathcal{X}_U/ U$. Then $\{\mathrm{Pic}^d_{X/K}\}_{d\in\mathbb{Z}}$ are torsors under $\mathrm{Pic}^0_{X/K}$. Since every rational section in abelian scheme over a smooth base necessarily extends to a regular section, the subgroup of canonical torsors $\langle[\mathrm{Pic}^1_{X/K}]\rangle\subseteq\mathrm{H}_{\mathrm{et}}^1(K,\mathrm{Pic}^0_{X/K})$ has the same order as over $U$. However, we show that there are uncountably many non-canonical torsors.

\begin{theorem}\label{thm2}
    The group of $\mathrm{Pic}^0_{X/K}$-torsors over $K$ is uncountable:
$$\mathrm{card}\left(\mathrm{H}^1_{\mathrm{et}}(K,\mathrm{Pic}^0_{X/K})\right)>\aleph_0$$
\end{theorem}
The proof is based on Theorem \ref{thm1}, the idea is to compare the Weil-Ch\^atelet group over $U$ with that of open parts in $U$, then analyze the difference in the colimit. Eventually we show that the term $\mathrm{H}^1_{\mathrm{et}}(K,\mathrm{Pic}^0_{X/K})$ fits into a finite exact sequence, where one other term has uncountable cardinality and all the rest are countable. For uncountablity, the key input is Lemma \ref{split}; for countability, one uses finiteness of \'etale cohomology with torsion coefficients and Artin vanishing.

\subsection{Parallel results on
    the generic genus-\texorpdfstring{$g$}{g} curve}
%We apply our arguments to the generic genus $g$ curves. Let $g\geq3$ be an integer, let $\mathcal{M}_g$ be the Deligne-Mumford stack of moduli of smooth genus-$g$ curves. 
Let $\pi\colon\mathcal{M}_{g,1}\to\mathcal{M}_g$ be the universal family og genus-$g$ curves. Let $K$ be the function field of $U$ and let $C/K$ be the generic fiber of $\pi$.
We pose the following questions:
\begin{question}\label{que1} Is every torsor of the universal Jacobian necessarily canonical: $$\mathrm{H}^1_{\mathrm{et}}(\mathcal{M}_g,\mathrm{Pic}^0_{\mathcal{M}_{g,1}/\mathcal{M}_g})=\langle[\mathrm{Pic}^1_{\mathcal{M}_{g,1}/\mathcal{M}_g}]\rangle?$$
\end{question}
\begin{question}\label{que2}Over the generic point, are there uncountably many non-canonical torsors $$\mathrm{card}\left(\mathrm{H}^1_{\mathrm{et}}(K,\mathrm{Pic}^0_{C/K})\right)>\aleph_0?$$
\end{question}
%While we could not answer them, 
We reduce Question \ref{que1} to injection of some third group cohomology between mapping class groups, see Proposition \ref{part-res-1}. We also prove a related result: 
\begin{theorem}\label{thm3} For $g\geq4,n\geq1$, every torsor of the Jacobian of the universal $n$-marked curve is trivial
$$\mathrm{H}^1_{\mathrm{et}}(\mathcal{M}_{g,n},\mathrm{Pic}^0_{\mathcal{M}_{g,{n+1}}/\mathcal{M}_{g,n}})=0.$$
\end{theorem}
One can work on Question \ref{que2} with same strategy as in Theorem \ref{thm2}.
However, on $\mathcal{M}_g$, we do not have vanishing results as in the proof of Theorem \ref{thm2}. 
Some non-vanishing terms are inessential, as they are uniformly bounded by a fixed countable abelian group when we take direct limit to the generic point. The only difficulty is that the Artin vanishing $\mathrm{H}^3(U\backslash\overline{Z},\mathbb{Q}_l/\mathbb{Q}_l(1))=0$ used in Lemma \ref{lemma-10-vanishing} does not generalize to $\mathrm{H}^3(\mathcal{M}_{g}\backslash\overline{Z},\mathbb{Q}_l/\mathbb{Z}_l)=0$. 
In Proposition \ref{part-res-2},
we give a sufficient condition for an affirmatively answer to Question \ref{que2}. The sufficient condition might be too strong to hold, or Question \ref{que2} might just have a negative answer, i.e., $\mathrm{H}^1_{\mathrm{et}}(K,\mathrm{Pic}^0_{C/K})$ could be countable or even finite, that would be a bit surprising.

\section{Preparation}
The universal family of degree-$m$ Fermat curves $\mathrm{pr_1}\colon\mathcal{X}\to\mathbb{P}^2_{[a_0:a_1:a_2]}$ is the union of the subfamily of smooth Fermat curves and the subfamily of singular curves. We show the smooth part is isotrivial and explicitly calculate the monodromy action. We give explicit stratification of the singular part. Then we recall the ramification sequence for Brauer group -- the key tool for the proof of our main theorems.

\subsection{The smooth part: Isotriviality and the monodromy action}We show that the universal family of smooth Fermat curves $(\mathrm{pr}_1)_U\colon\mathcal{X}_U\to U$ is isotrivial, and spell out its monodromy action. Let $F_m$ be the standard Fermat curve in $\mathbb{P}^2_{[X_0:X_1:X_2]}$ defined by $$X_0^m+X_1^m+X_2^m=0.$$
\begin{lemma} Let $\widetilde{U}=U$, and let $f\colon\widetilde{U}\to U$ be the map induced by $[a_0:a_1:a_2]\mapsto [a_0^m:a_1^m:a_2^m]$,
then there is an isomorphism $\mathcal{X}_U\times_U\widetilde{U}\cong F_m\times_{\mathbb{C}}\widetilde{U}$ over $\widetilde{U}$. 
\end{lemma}
\begin{proof}
The family $\mathcal{X}_U\times_U\widetilde{U}\subset\mathbb{P}^2_{\widetilde{U}}$ is defined by $a_0^mX_0^m+a_1^mX_1^m+a_2^mX_2^m=0$, therefore it is $\widetilde{U}$-isomorphic to $F_m\times_{\mathbb{C}}\widetilde{U}\subset\mathbb{P}^2_{\widetilde{U}}$ via $[X_0:X_1:X_2]\mapsto [a_0X_0:a_1X_1:a_2X_2]$.
\end{proof}
From now on, let us use $\phi$ to denote 
the $\widetilde{U}$-isomorphism 
$\phi\colon\mathcal{X}_{\widetilde{U}}\to C\times_\mathbb{C}\widetilde{U}$ defined by $$[X_0:X_1:X_2]\mapsto[a_0X_0:a_1X_1:a_2X_2].$$

Note that $f\colon\widetilde{U}\to U$ is a finite Galois cover, with deck transformation group $\mathrm{Aut}_U(\widetilde{U})\cong\mu_m^{\oplus2}.$ We choose our notations so that for any
$\sigma:=(\zeta_0,\zeta_1)\in\mu_m^{\oplus2}$, the action $a_\sigma\colon \widetilde{U}\to\widetilde{U}$ pulls back the coordinates on $\widetilde{U}$ by $[a_0:a_1:a_2]\mapsto [\zeta_0a_0:\zeta_1a_1:a_2].$
\begin{lemma}\label{mondromy-action}The family $(\mathrm{pr}_1)_U\colon\mathcal{X}_U\to U$ can be constructed as the quotient of projection $\mathrm{pr}_1'\colon F_m\times_{\mathbb{C}}\widetilde{U}\to \widetilde{U}$ under the equivariant $\mu_m^{\oplus2}$-action, where $(\zeta_0,\zeta_1)\in\mu_m^{\oplus2}$ acts on $F_m\times_{\mathbb{C}}\widetilde{U}$ via $$([X_0:X_1:X_2],[a_0:a_1:a_2])\mapsto ([\zeta_0X_0:\zeta_1X_1:X_2],[\zeta_0a_0:\zeta_1a_1:a_2])$$  and on $\widetilde{U}$ via $[a_0:a_1:a_2]\mapsto[\zeta_0a_0:\zeta_1a_1:a_2]$.
\end{lemma}
\begin{proof}
For $\sigma:=(\zeta_0,\zeta_1)\in\mu_m^{\oplus2}$, let $a_{\sigma}\in\mathrm{Aut}_U(\widetilde{U})$ be the deck transformation. Let  
$(a_\sigma)_{\mathcal{X}_U}\in\mathrm{Aut}_{\mathbb{C}}(\mathcal{X}_{\widetilde{U}})$ 
be the base change of $a_\sigma$ along $(\mathrm{pr}_1)_U\colon\mathcal{X}_U\to U$. Note that $\mathcal{X}_U$ is constructed as the quotient of $F_m\times_\mathbb{C}\widetilde{U}$ by $\mathrm{Aut}_U(\widetilde{U})$ via $\sigma\mapsto \phi\circ(a_\sigma)_{\mathcal{X}_U}\circ\phi^{-1}$. Then one directly computes the composition $([X_0:X_1:X_2],[a_0:a_1:a_2])\mapsto ([a_0^{-1}X_0:a_1^{-1}X_1:a_2^{-1}X_2],[a_0:a_1:a_2])\mapsto([a_0^{-1}X_0:a_1^{-1}X_1:a_2^{-1}X_2],[\zeta_0a_0:\zeta_1a_1:a_2])\mapsto([\zeta_0X_0:\zeta_1X_1:X_2],[\zeta_0a_0:\zeta_1a_1:a_2])$.
\end{proof}

Under the trivialization $\phi\colon\mathcal{X}_{\widetilde{U}}\to F_m\times_{\mathbb{C}}\widetilde{U}$ the the fibers over $\widetilde{U}$ are canonically identified: For any fixed point $P=[a_0:a_1:a_2]\in \widetilde{U}$, the $\mathrm{Aut}_U(\widetilde{U})$-action maps the fiber over $P$ to the fiber over $\sigma(P)$ via $$\rho(\sigma)\colon F_m\to F_m,\  [X_0:X_1:X_2]\mapsto[\zeta_0X_0:\zeta_1X_1:X_2].$$ 

Let us call $\rho\colon\mathrm{Aut}_U(\widetilde{U})\to\mathrm{Aut}_{\mathbb{C}}(F_m),\ \sigma\mapsto \rho(\sigma)$ the monodromy action of the isotrivial family with respect to the trivialization $\phi$. In general, different choices for $\phi$ alters $\rho$ by an inner automorphism of $\mathrm{Aut}_U(\widetilde{U})$. In our case, since $\mathrm{Aut}_U(\widetilde{U})$ is an abelian group, any other trivialization of $\mathcal{X}_{\widetilde{U}}$ gives rise to the same monodromy action.

\subsection{The ramification sequence}
The following proposition is the basis for our calculation:
\begin{lemma}{\cite[Theorem 3.18]{CT-Brauer-Grothendieck-Group}}\label{ramificaiton-lemma} Let $k$ be a field and let $X$ be a $k$-variety. Let $Z\subset X$ be a smooth closed subscheme of pure codimension $c$. Let $U=X\backslash Z$ be the open complement. Then for any prime $l\neq\mathrm{char}(k)$, the following holds
\begin{enumerate}
\item If $c\geq2$, then the restriction map $\mathrm{Br}(X)[l^{\infty}]\to\mathrm{Br}(U)[l^{\infty}]$ is an isomorphism.
\item If $c=1$ and $D_1,\cdots,D_m$ are the connected components of $Z$, then we have exact sequence
$$\xymatrix{
0\ar[r]&\mathrm{Br}(X)[l^{\infty}]\ar[r]&\mathrm{Br}(U)[l^{\infty}]\ar[r]&\bigoplus_{i=1}^m\mathrm{H}_{\mathrm{et}}^1(D_i,\mathbb{Q}_l/\mathbb{Z}_l)\\
&\ar[r]&\mathrm{H}^3_{\mathrm{et}}(X,\mathbb{Q}_l/\mathbb{Z}_l(1))\ar[r]&\mathrm{H}^3_{\mathrm{et}}(U,\mathbb{Q}_l/\mathbb{Z}_l(1))
}$$    
\end{enumerate}
\end{lemma}
We briefly explain the notations. For an abelian group $M$, we denote the subgroup of $n$-torsion elements by $M[n]$. Let $l$ be a prime, we define the $l$-power torsion subgroup of $M$ by
$M[l^\infty]=\bigcup_{r\geq1}M[l^r]$. For a torsion abelian group $M$, one has $M[l^\infty]\cong M\otimes_{\mathbb{Z}}\mathbb{Z}[1/l]$, and one has primary decomposition $$M=\bigoplus_{l\textrm{\ prime}}M[l^{\infty}]=\bigoplus_{l\textrm{\ prime}}(M\otimes_{\mathbb{Z}}\mathbb{Z}[1/l]).$$ By exactness of localization, for any submodule $N\subset M$, one has $(M/N)[l^{\infty}]=M[l^{\infty}]/N[l^{\infty}]$. On the finite level, in general $(M/N)[l^r]\ncong M[l^r]/N[l^r]$.
Here the group $\mathbb{Q}_l/\mathbb{Z}_l(1)$ is defined as union of $l$-primary roots unity
$\mathrm{colim}_{r\geq1}\mu_{l^r}$. Since cohomology commutes with directed colimit, one has $\mathrm{H}_{\mathrm{et}}^i(X,\mathbb{Q}_l/\mathbb{Z}_l(1))\cong\mathrm{colim}_{r\geq1}\mathrm{H}_{\mathrm{et}}^i(X,\mu_{l^r})$. The ramification is obtained by splicing Gysin sequence with the long exactness of Kummer sequence. One needs the condition $l\neq\mathrm{char}(k)$ both for the Kummer sequence and for the excision sequence, where purity $i^!\mathcal{F}\cong i^*\mathcal{F}[-2c](-c)$ is used. We will not worry about these since we only work over $\mathbb{C}$.

\section{Proof of Theorem \ref{thm1}}
We will play around several short exact sequences to calculate $\mathrm{H}^1(U,\mathrm{Pic}^0_{\mathcal{X}_U/U})$.
\subsection{Short exact sequence of connected components}
Recall that $\mathrm{Pic}_{\mathcal{X}_U/U}$ is the relative Picard scheme of $(\mathrm{pr}_1)_U\colon\mathcal{X}_U\to U$ and $\mathrm{Pic}^0_{\mathcal{X}_U/U}$ is the kernel of the degree map.
\begin{lemma}\label{lem-conn-components} 
We have short exact sequence 
$$\xymatrix{\mathrm{Pic}_{\mathcal{X}_U/U}(U)\ar[r]^-{\deg_U}&\mathbb{Z}\ar[r]^-{\delta}&\mathrm{H}_{\mathrm{et}}^1(U,\mathrm{Pic}^0_{\mathcal{X}_U/U})\ar[r]&\mathrm{H}_{\mathrm{et}}^1(U,\mathrm{Pic}_{\mathcal{X}_U/U})\ar[r]&0,}$$
the map $\delta$ sends $d\in\mathbb{Z}$ to the class of $\mathrm{Pic}^d_{\mathcal{X}_U/U}$.
\end{lemma}
\begin{proof}
Consider the exact sequence of \'etale sheaves on $U$:
$$\xymatrix{
0\ar[r]&\mathrm{Pic}^0_{\mathcal{X}_U/U}\ar[r]&\mathrm{Pic}_{\mathcal{X}_U/U}\ar[r]^-{\deg}&\underline{\mathbb{Z}_U}\ar[r]&0
},$$
where $\underline{\mathbb{Z}_U}$ is the constant \'etale sheaf with stalks $\mathbb{Z}$ on $U$. Note that $\mathrm{H}^1_{\mathrm{et}}(U,\underline{\mathbb{Z}_U})=\mathrm{Hom}_{\mathrm{cont}}(\pi_1^{\mathrm{et}}(U),\mathbb{Z})=0$ because
since $\pi_1^{\mathrm{et}}(U)$ is a profinite group, we conclude first part of the lemma from the associated long exact sequence. 
For the description of the connecting homomorphism $\delta$, it suffices to show $1$ maps to $[\mathrm{Pic}^1_{\mathcal{X}_U/U}]$. Let us represent the degree one line bundle by $\mathcal{O}_{\mathcal{X}_{U'}}(S')$, where $S'$ is a section of some \'etale base change $(\mathrm{pr}_1)_{U'}\colon\mathcal{X}_{U'}\to U'$. A different choice $S''$ alters the line bundle by $\mathcal{O}_{\mathcal{X}_{U'}}(S''-S')\in\mathrm{Pic}^0_{\mathcal{X}_{U'}}(U')$. This yields the cocycle of the torsor $\mathrm{Pic}^1_{\mathcal{X}_U/U}$.
\end{proof}
Now Theorem \ref{thm1} will follow once we show that 
\begin{enumerate}
\item The \'etale cohomology group $\mathrm{H}_{\mathrm{et}}^1(U,\mathrm{Pic}_{\mathcal{X}_U/U})=0$ \item The image of $\mathrm{deg}_U\colon\mathrm{Pic}_{\mathcal{X}_U/U}(U)\to\mathbb{Z}$ is $m\mathbb{Z}$ when $2\nmid m$ and $(m/2)\mathbb{Z}$ when $2\mid m$.
\end{enumerate}
\subsection{The Leray spectral sequence}\label{Leray} Let us consider the Leray spectrals sequence for $\mathbb{G}_{m,\mathcal{X}_U}$ along $(\mathrm{pr}_1)_U\colon\mathcal{X}_U\to U$. The low degree terms on $E_2$ page fits into exact sequence
$$\xymatrix{0\ar[r]&\frac{\mathrm{Pic}_{\mathcal{X}_U/U}(U)}{\mathrm{Pic}(\mathcal{X}_U)}\ar[r]&\mathrm{Br}(U)\ar[r]&\mathrm{Br}(\mathcal{X}_U)\ar[r]&\mathrm{H}^1(U,\mathrm{Pic}_{\mathcal{X}_U/U})\ar[r]&\mathrm{H}^3(U,\mathbb{G}_m)}$$
The leftmost term is injection because $\mathrm{Pic}(U)=0$.
Now note that Brauer groups of regular schemes and \'etale cohomology group $\mathrm{H}^3_{\mathrm{et}}(U,\mathbb{G}_m)$ are torsion, the long exact sequence consists of torsion abelian groups, and can be calculated on the $l$-primary components:
$$\xymatrix{0\ar[r]&\left(\frac{\mathrm{Pic}_{\mathcal{X}_U/U}(U)}{\mathrm{Pic}(\mathcal{X}_U)}\right)[l^\infty]\ar[r]&\mathrm{Br}(U)[l^\infty]\ar[r]&\mathrm{Br}(\mathcal{X}_U)[l^\infty]\ar[r]&\mathrm{H}^1(U,\mathrm{Pic}_{\mathcal{X}_U/U})[l^\infty]\ar[r]&0}$$
The last term vanishes because by Kummer sequence, $\mathrm{H}^3(U,\mathbb{G}_m)[l^r]$ is a quotient of $\mathrm{H}^3(U,{\mu_{l^r}})$, which equals to $0$ by Artin vanishing. Next we determine the kernel and cokernel of $(\mathrm{pr}_1)_U^*\colon\mathrm{Br}(U)[l^\infty]\to\mathrm{Br}(\mathcal{X}_U)[l^\infty]$ with the help of ramification sequence of Brauer groups.

\subsection{Comparison of the ramification sequences}\label{comparison-section}

For $i=0,1,2$, let 
\begin{enumerate}
\item $E_i=V(a_{i-1},a_{i+1})\cap\mathcal{X}$ be the non-reduced fibers in the family $\mathrm{pr}_1\colon\mathcal{X}\to\mathbb{P}^2_{[a_0:a_1:a_2]}$,
\item $D_i=V(a_i)$ be the inverse image of coordinate axes in $\mathbb{P}^2_{[a_0:a_1:a_2]}$,
\item $F_i=V(X_{i-1},X_{i+1})\cap \mathcal{X}$, they are the singular locus of $D_i$,
\item $D_i^\circ=D_i-(E_{i-1}\cup E_{i+1}\cup F_i)$ be the open part of $D_i$,
\item $L_i=\mathrm{pr}_1(D_i)$, $L_i^\circ=\mathrm{pr}_1(D_i^\circ)$ and $P_i=\mathrm{pr}_1(E_i)$.
\end{enumerate} 
Let $E=\bigcup_{i=0}^2E_i,F=\bigcup_{i=0}^2F_i,
%D=\bigcup_{i=0}^2D_i,
D^\circ=\bigcup_{i=0}^2D_i^\circ,%L=\bigcup_{i=0}^2L_i,
L^\circ=\bigcup_{i=0}^2L_i^\circ$ and $P=\bigcup_{i=0}^2P_i$
, then $$\mathcal{X}=\mathcal{X}_U\sqcup D^\circ\sqcup( E\cup F),\ \  \mathbb{P}^2_{[a_0:a_1:a_2]}=U\sqcup L^\circ\cup P. $$ 
Now the ramification sequences on $U$ reads
$$\xymatrix{
0\ar[r]&\mathrm{Br}(\mathbb{P}^2-P)[l^\infty]\ar[r]&\mathrm{Br}(U)[l^{\infty}]\ar[r]&\bigoplus_{i=0}^2\mathrm{H}_{\mathrm{et}}^1(L^\circ_i,\mathbb{Q}_l/\mathbb{Z}_l)\\
&\ar[r]&\mathrm{H}_{\mathrm{et}}^3(\mathbb{P}^2-P,\mathbb{Q}_l/\mathbb{Z}_l(1)
)\ar[r]&\mathrm{H}_{\mathrm{et}}^3(U,\mathbb{Q}_l/\mathbb{Z}_l(1))
}$$
and the ramification on $\mathcal{X}_U$ reads
$$\xymatrix{
0\ar[r]&\mathrm{Br}(\mathcal{X}-E\cup F)[l^{\infty}]\ar[r]&\mathrm{Br}(\mathcal{X}_U)[l^{\infty}]\ar[r]&\bigoplus_{i=0}^2\mathrm{H}_{\mathrm{et}}^1(D^\circ_i,\mathbb{Q}_l/\mathbb{Z}_l)\\
&\ar[r]&\mathrm{H}_{\mathrm{et}}^3(\mathcal{X}-E\cup F,\mathbb{Q}_l/\mathbb{Z}_l(1))\ar[r]^-{\alpha}&\mathrm{H}_{\mathrm{et}}^3(\mathcal{X}_U,\mathbb{Q}_l/\mathbb{Z}_l(1)).
}$$

\begin{lemma}We have $\mathrm{Br}(\mathbb{P}^2-P)=0$ and $\mathrm{Br}(\mathcal{X}-E\cup F)=0.$
\end{lemma}
\begin{proof}Applying Lemma \ref{ramificaiton-lemma}(1), we see that $\mathrm{Br}(\mathbb{P}^2-P)=\mathrm{Br}(\mathbb{P}^2)$ and that $\mathrm{Br}(\mathcal{X}-E\cup F)=\mathrm{Br}(\mathcal{X}-E\cap F)=\mathrm{Br}(\mathcal{X})$. Then note that Brauer group is a stable birational invariant for smooth projective varieties and that $\mathrm{Br}(\mathbb{C})=0$.
\end{proof}

\subsection{Generators of cohomology of punctured spaces and their pullbacks} %We fix a generator of $\mathrm{H}^1_{\mathrm{et}}(L_i^\circ,\mathbb{Q}_l/\mathbb{Z}_l)$: 
Consider the inclusion $L_i^\circ\subset\mathbb{P}^2-P_{i-1}\cup P_{i+1}$, the ramification sequence gives us canonical isomorphism 
\begin{align*}
\mathrm{H}^1_{\mathrm{et}}(L_i^\circ,\mathbb{Q}_l/\mathbb{Z}_l)&\cong\mathrm{H}^3(\mathbb{P}^2-P_{i+1}\cup P_{i-1},\mathbb{Q}_l/\mathbb{Z}_l(1))\\
&\cong\mathrm{H}^0(P_{i+1}\cup P_{i-1},\mathbb{Q}_l/\mathbb{Z}_l(1))\\
&\cong\mathbb{Q}_l/\mathbb{Z}_l(-1)\left([1_{P_{i+1}}]-[1_{P_{i-1}}]\right)
\end{align*}
we call $[1_{P_{i+1}}]-[1_{P_{i-1}}]$ the canonical generator for $\mathrm{H}^1_{\mathrm{et}}(L_i^\circ,\mathbb{Q}_l/\mathbb{Z}_l).$

For the next two lemmas, let $\widetilde{L_i^\circ}=L_i^\circ$ and let $(\mathrm{pr}_1)_{L_i^\circ}\colon \widetilde{L_i^\circ}\to {L_i^\circ}$ be the map induced by $$[a_0:a_1:a_2]\mapsto [a_0^m:a_1^m:a_2^m].$$ \begin{lemma}\label{pullback-cyclic}The pullback map induces multiplication by $m$ on $\mathrm{H}^1_{\mathrm{et}}(L_i^\circ,\mathbb{Q}_l/\mathbb{Z}_l)$: $$(\mathrm{pr}_1)_{L_i^\circ}\colon\mathrm{H}_{\mathrm{et}}^1(L_i^\circ,\mathbb{Q}_l/\mathbb{Z}_l)\to\mathrm{H}^1_{\mathrm{et}}(L_i^\circ,\mathbb{Q}_l/\mathbb{Z}_l),\ \  a\mapsto m\cdot a.$$
\end{lemma}
\begin{proof} Note that $\mathrm{H}^1_{\mathrm{et}}(L_i^\circ,\mathbb{Q}_l/\mathbb{Z}_l)\cong \mathrm{Hom}(\pi_{1}^{\mathrm{et}}(L_i^\circ),\mathbb{Q}_l/\mathbb{Z}_i)$ and that the induced map on \'etale fundmamental group is multiplication by $m$.
\end{proof}
\begin{lemma}
The scheme $D_i^\circ$ is canonically isomorphic to $\widetilde{L_i^\circ}\times\mathbb{A}^1$ as schemes over $L_i^\circ$, and the following map induced by pullback is multiplication by $m$: $$(\mathrm{pr}_1)_{L_i^\circ}^*\colon\mathrm{H}^1_{\mathrm{et}}(L_i^\circ,\mathbb{Q}_l/\mathbb{Z}_l)\to\mathrm{H}^1_{\mathrm{et}}(D_i^\circ,\mathbb{Q}_l/\mathbb{Z}_l)\cong\mathrm{H}^1_{\mathrm{et}}(\widetilde{L_i^\circ}\times\mathbb{A}^1,\mathbb{Q}_l/\mathbb{Z}_l)\cong\mathrm{H}_{\mathrm{et}}^1(L_i^\circ,\mathbb{Q}_l/\mathbb{Z}_l).$$ 
\end{lemma}
\begin{proof}
Without loss of generality, let $i=2$. Let us pick $\zeta\in\mathbb{C}$ so that $\zeta^m=-1$. Then we define the morphism $\alpha\colon \widetilde{L_2^\circ}\times\mathbb{A}^1\to D_2^\circ$ over $L_2^\circ$, given by $([a_0:a_1:0],[X_0:X_1:X_2])\mapsto ([\zeta X_1:X_0:0], X_1/X_2).$ The inverse isomorphism $\beta\colon D_2^\circ\to \widetilde{L_2^\circ}\times\mathbb{A}^1$ is given by $([a_0:a_1:0],t)\mapsto([a_0^m:a_1^m:0],[\zeta a_1:a_0:a_0t^{-1}]).$ One directly checks that the morphism $\alpha,\beta$ are compatible with the morphisms to $L_2^\circ$. Then we conclude by Lemma \ref{pullback-cyclic}
\end{proof}

By the Gysin sequence for embedding $\iota_P\colon P\hookrightarrow\mathbb{P}^2$, we have short exact sequence 
$$\xymatrix{0\ar[r]&\mathrm{H}^3_{\mathrm{et}}(\mathbb{P}^2-P,\mathbb{Q}_l/\mathbb{Z}_l)\ar[r]&\mathrm{H}_{\mathrm{et}}^0(P,\mathbb{Q}_l/\mathbb{Z}_l(-1))\ar[r]^{\iota_{P,*}}&\mathrm{H}_{\mathrm{et}}^4(\mathbb{P}^2,\mathbb{Q}_l/\mathbb{Z}_l(1))}$$
Note that $\mathrm{H}^0_{\mathrm{et}}(P,\mathbb{Q}_l/\mathbb{Z}_l(-1))\cong\bigoplus_{i=0}^2\mathbb{Q}_l/\mathbb{Z}_l(-1)[1_{P_i}]$ where $1_{P_i}$ is the indicator function on $P_i$, then \begin{align*}
\mathrm{H}^3_{\mathrm{et}}(\mathbb{P}^2-P,\mathbb{Q}_l/\mathbb{Z}_l)&=\mathrm{ker}(\iota_{P,*})\cong\mathbb{Q}_l/\mathbb{Z}_l(-1)^{\oplus2}\\
&\cong\mathbb{Q}_l/\mathbb{Z}_l(-1)\left([1_{P_2}]-[1_{P_0}]\right)\oplus \mathbb{Q}_l/\mathbb{Z}_l\left([1_{P_1}]-[1_{P_0}]\right)\end{align*}
Similarly, running the Gysin sequence on $\mathcal{X}-E\cup F$, we see that \begin{align*}
\mathrm{H}^3(\mathcal{X}-E\cup F,\mathbb{Q}_l/\mathbb{Z}_l(1))&\cong\mathbb{Q}_l/\mathbb{Z}_l(-1)^{\oplus4}
\\
&\cong\bigoplus_{i=1}^2\mathbb{Q}_l/\mathbb{Z}_l(-1)([1_{E_i^\circ}]-[1_{E_0^\circ}])\oplus\bigoplus_{i=1}^2\mathbb{Q}_l/\mathbb{Z}_l(-1)([1_{F_i^\circ}]-[1_{F_0^\circ}])
\end{align*} is identified under the Gysin map as a direct summand of \begin{align*}
\mathrm{H}^0(E\cup F-E\cap F,\mathbb{Q}_l/\mathbb{Z}_l(-1))&\cong \mathbb{Q}_l/\mathbb{Z}_l(-1)^{\oplus6}
\\
&\cong\bigoplus_{i=0}^2\mathbb{Q}_l/\mathbb{Z}_l(-1)[1_{E_i^\circ}]\oplus\bigoplus_{i=0}^2\mathbb{Q}_l/\mathbb{Z}_l(-1)[1_{F_i^\circ}]
\end{align*}

Now we compare the ramification sequences in Section \ref{comparison-section}, and get the following diagram:

$$\xymatrix{
    0\ar[r]&\mathrm{Br}(U)[l^\infty]\ar[d]^{(\mathrm{pr}_1)_U^*}\ar[r]&\bigoplus_{i=0}^2\mathrm{H}^1(L_i^\circ,\mathbb{Q}_l/\mathbb{Z}_l)\ar[d]\ar[r]&\mathrm{H}^3(\mathbb{P}^2-P,\mathbb{Q}_l/\mathbb{Z}_l(1))\ar[d]\ar[r]&0\\
    0\ar[r]&\mathrm{Br}(\mathcal{X}_U)[l^{\infty}]\ar[r]&\bigoplus_{i=0}^2\mathrm{H}^1(D_i^\circ,\mathbb{Q}_l/\mathbb{Z}_l)\ar[r]&\mathrm{H}^3(\mathcal{X}-E\cup F,\mathbb{Q}_l/\mathbb{Z}_l(1))
}$$
\begin{lemma}\label{surjection}
The pullback map $(\mathrm{pr}_1)_U^*\colon\mathrm{Br}(U)\to\mathrm{Br}(\mathcal{X}_U)$ is surjective with kernel  $\mu_m^{-1}$.
\end{lemma}
\begin{proof}
Since Brauer groups are torsion, it suffices to calculate the $l$-primary parts. By previous discussion, we can rewrite the commutative diagram as 
$$\xymatrix{
    0\ar[r]&\mathrm{Br}(U)[l^{\infty}]\ar[r]\ar[d]^{(\mathrm{pr}_1)_U^*[l^\infty]}&\left(\mathbb{Q}_l/\mathbb{Z}_l(-1)\right)^{\oplus3}\ar[r]^{\iota_1}\ar[d]^{m}&(\mathbb{Q}_l/\mathbb{Z}_l(-1))^{\oplus2}\ar[d]^{m\oplus 0}\ar[r]&0\\
    0\ar[r]&\mathrm{Br}(\mathcal{X}_U)[l^{\infty}]\ar[r]&\left(\mathbb{Q}_l/\mathbb{Z}_l(-1)\right)^{\oplus3}\ar[r]^{\iota_2}&(\mathbb{Q}_l/\mathbb{Z}_l(-1))^{\oplus4}
}$$
The Gysin map $\iota_1$ sends  the canonical generator of $\mathrm{H}^1(L_i^\circ,\mathbb{Q}_l/\mathbb{Z}_l)$ to $[1_{P_{i+1}}]-[1_{P_{i-1}}]$, and the Gysin map $\iota_2$ sends the canonical generator of $\mathrm{H}^1(D_i^\circ,\mathbb{Q}_l/\mathbb{Z}_l)\cong\mathrm{H}^1(\widetilde{L}_i^\circ\times\mathbb{A}^1,\mathbb{Q}_l/\mathbb{Z}_l)$ to $[1_{E_{i+1}^\circ}]-[1_{E_{i-1}^\circ}]$.
The rightmost vertical map $m\oplus0$ sends the generator $[1_{P_i}]-[1_{P_0}]$ to $m([1_{F_i}]-[1_{F_0}])$ by commutativity of the diagram. By snake lemma, we get long exact sequence
$$\xymatrix{0\ar[r]&(\mathrm{ker}(\mathrm{pr}_1)_U^*)[l^\infty]\ar[r]&\mu_{(m,l^\infty)}^{\oplus3}\ar[r]&\mu_{(m,l^\infty)}^{\oplus2}\ar[r]&(\mathrm{coker}(\mathrm{pr}_1)_U^*)[l^\infty]\ar[r]&0}$$
The middle map is surjection, as one checks on the generators. We conclude by summing up the $l$-primary parts.
\end{proof}

Now from the above discussion we conclude that
\begin{proposition}\label{vanishing-whole-torsor}
The \'etale cohomology group $\mathrm{H}^1_{\mathrm{et}}(U,\mathrm{Pic}_{\mathcal{X}_U/U})=0$.
\end{proposition}
\begin{proof}
Use the long exact sequence in Section \ref{Leray} and Lemma \ref{surjection}. 
\end{proof}

\subsection{Period of the torsors}
Recall that the Leray spectral sequence give us exact sequence $$\xymatrix{
0\ar[r]&\mathrm{Pic}(\mathcal{X}_U)\to\mathrm{Pic}_{\mathcal{X}_U/U}(U)\ar[r]^-{\mathrm{d}_2^{0,1}}&\mathrm{Br}(U)\ar[r]&\mathrm{Br}(\mathcal{X}_U)
}.$$
Since $\mathcal{X}$ is a projective bundle, $\mathrm{Pic}(\mathcal{X})\cong\mathbb{Z}\cdot\mathrm{pr}_1^*\mathcal{O}(1)\oplus\mathbb{Z}\cdot\mathrm{pr}_2^*\mathcal{O}(1)$. By excision sequence of class groups, one conclude that $\mathrm{Pic}(\mathcal{X}_U)\cong\mathbb{Z}\cdot\mathrm{pr}_2^*\mathcal{O}(1).$  
Together with Lemma \ref{surjection}, we rewrite the previous sequence as 
\begin{equation}\xymatrix{0\ar[r]&\mathbb{Z}\cdot\mathrm{pr}_2^*\mathcal{O}(1)\ar[r]&\mathrm{Pic}_{\mathcal{X}_U/U}(U)\ar[r]^-{\mathrm{d}_{2}^{0,1}}&\mathrm{\mu}_m^{-1}\ar[r]&0}\label{eqn1}\end{equation}
Recall we also have the following naive short exact sequence
\begin{equation}\xymatrix{0\ar[r]&\mathrm{Pic}^0_{\mathcal{X}_U/U}(U)\ar[r]&\mathrm{Pic}_{\mathcal{X}_U/U}(U)\ar[r]^-{\mathrm{deg}_U}&\mathrm{\mathbb{Z}}}\label{eqn2}\end{equation}

\begin{lemma}\label{invariants}The sections to the $m$-torsion subscheme of the Picard scheme is
$$\mathrm{Pic}^0_{\mathcal{X}_U/U}[m](U)=\left\{\begin{array}{cc}
\mu_m ,& 2\nmid m  \\
    \mu_{\frac{m}{2}}, & 2\mid m 
\end{array}\right.$$
\end{lemma}
\begin{proof}
Consider the Kummer sequence on $\mathcal{X}_U$. Long exact sequence of direct images gives $$\mathrm{H}^0(U,R^1(\mathrm{pr}_{1})_{U,*}\mu_m)\cong\mathrm{Pic}_{\mathcal{X}_U/U}[m](U).$$
Therefore it suffices to calculate the invariants of the local system $R^1(\mathrm{pr}_{1})_{U,*}\mu_m$. Since the family is isotrival, the left hand side is the same as invariants of the monodromy action on the standard Fermat curve $\mathrm{H}^1(F_m,\mu_m)\cong\mathrm{H}_1(F_m,\mathbb{Z}/n\mathbb{Z})\otimes\mu_n$. Recall that the $\mu_m\times\mu_m$ action on $\mathrm{H}_1(F_m,\mathbb{Z}/n\mathbb{Z})$ is given as follows:
\begin{enumerate}
\item Let 
$F_m=V(X_0^m+X_1^m+X_2^m)\subset\mathbb{P}^2$ be the standard Fermat curve. Let $F_m^\circ=F_m\backslash V(X_2)$ and $F_m^{\circ\circ}=F_m\backslash V(X_0X_1X_2)$ be the complement of coordinate axes. Then by \cite[Proposition 4]{Galois-action-homology-MR3596577}, the relative homology $H_1(F_m^\circ,F_m^\circ\backslash F_m^{\circ\circ};\mathbb{Z}/m\mathbb{Z})$ is canonically isomorphic to a rank one $(\mathbb{Z}/m\mathbb{Z})[t_0,t_1]/(t_0^m-1,t_1^m-1)$-module, therefore any element can be identified with $\sum_{0\leq i,j\leq m-1}a_{ij}t_0^it_1^j$, or a $m\times m$ matrix $A=(a_{ij})_{m\times m}$. Exact sequence of relative homology gives us embedding $\mathrm{H}_1(F_m^\circ,\mathbb{Z}/m\mathbb{Z})\subseteq\mathrm{H}_1(F_m^\circ,F_m^\circ\backslash F_m^{\circ\circ};\mathbb{Z}/m\mathbb{Z})$,
an element come from $\mathrm{H}_1(F_m^\circ,\mathbb{Z}/m\mathbb{Z})$ if and only if the rows and columns sum up to zero, i.e. $$\mathrm{H}_1(F_m^\circ,\mathbb{Z}/m\mathbb{Z})\cong \{(1,\cdots,1)_{1\times m}A_{m\times m}=
(1,\cdots,1)_{1\times m}A_{m\times m}^T=\mathbf{0}_{1\times m}\}
$$
Since the basis consists of $t_0^it_1^j$, multiplication by $(t_0,1),(1,t_1)\in\mu_m\times\mu_m$ corresponds to left and right multiplication of a matrix by $J=\left(\begin{array}{cccc}
    0 & I_{m-1}\\
1&\mathbf{0}_{1\times(m-1)}
\end{array}\right)$.
\item The homology of $F_m$ is calculated by exact sequence of relative homology $$\xymatrix{\mathrm{H}_2(F_m,F_m^\circ;\mathbb{Z}/m\mathbb{Z})\ar[r]^-{D}&\mathrm{H}_1(F_m^\circ,\mathbb{Z}/m\mathbb{Z})\ar[r]&\mathrm{H}_1(F_m,\mathbb{Z}/m\mathbb{Z})\ar[r]&0}.$$
By \cite[Proposition 5]{Galois-action-homology-MR3596577}, the image of $D$ consist of all elements such that $a_{i+1,j+1}=a_{ij}$, where the indices $i,j$ take sum in $\mathbb{Z}/m\mathbb{Z}$, under the previous identification of $\mathrm{H}_1(F_m^\circ,\mathbb{Z}/m\mathbb{Z})=0$ with matrices: $$\mathrm{Im}(D)=\{a_0I+a_1J+\cdots+a_{m-1}J^{m-1}|a_0+\cdots+a_{m-1}=0\}.$$
\end{enumerate}
Our goal is to determine the group of $\mu_m\times\mu_m$-invariants of $\mathrm{H}_1(F_m,\mathbb{Z}/m\mathbb{Z})=\mathrm{H}_1(F_m^\circ,\mathbb{Z}/m\mathbb{Z})/\mathrm{Im}(D)$.
Now for any element in $\mathrm{H}_1(F_m,\mathbb{Z}/m\mathbb{Z})$ represented by matrix $A=(a_{ij})$, suppose that it is $\mu_m\times\mu_m$-invariant, then $(I-J)A,A(I-J)\in\mathrm{Im}(D)$. Explicitly writing down the equations, we see that modulo $\mathrm{Im}(D)$, the matrix $A$ is determined by the diagonal terms $a_{00},\cdots,a_{m-1,m-1}$, subject to the relation that $a_{ii}-a_{i+1,i+1}$ are all equal to some fixed $d\in\mathbb{Z}/m\mathbb{Z}$. The condition $(1,\cdots,1)A=0$ turns out to be $\frac{m(m-1)}{2}d=0\in\mathbb{Z}/m\mathbb{Z}$. When $m$ is odd, there is no restriction on $d$, however when $m$ is even, $d$ can only take values in $2\mathbb{Z}/m\mathbb{Z}$. 
\end{proof}
\begin{proposition}\label{image-deg}
The image of $\mathrm{deg}_U\colon\mathrm{Pic}_{\mathcal{X}_U/U}(U)\to\mathbb{Z}$ is 
$$\mathrm{Im}(\mathrm{deg}_U)=\left\{\begin{array}{cc}
    m\mathbb{Z} & 2\mid m \\
(\frac{m}{2})\mathbb{Z}& 2\nmid m 
\end{array}\right.$$
\end{proposition}
\begin{proof}
Note that $\mathbb{Z}\cdot\mathrm{pr}_2^*\mathcal{O}(1)$ intersects $\mathrm{Pic}^0_{\mathcal{X}_U/U}(U)$ trivially, from sequence (\ref{eqn1}) we know $\mathrm{Pic}^0_{\mathcal{X}_U/U}(U)$ maps injectively into $\mu_m^{-1}$ via $\mathrm{d}_2^{0,1}$, thus $\mathrm{Pic}^0_{\mathcal{X}_U/U}(U)=\mathrm{Pic}^0_{\mathcal{X}_U/U}(U)[m].$
Let us note that $\mathrm{Pic}^0_{\mathcal{X}_U/U}(U)[m]=\mathrm{Pic}^0_{\mathcal{X}_U/U}[m](U)$, because equalizers commute. Now let us consider the subgroup $$\mathbb{Z}\cdot\mathrm{pr}_1^*\mathcal{O}(1)\oplus\mathrm{Pic}^0_{\mathcal{X}_U/U}(U)\subset\mathrm{Pic}_{\mathcal{X}_U/U}(U).$$
\begin{enumerate}
\item When $2\nmid m$, by Lemma \ref{invariants}, the subgroup $\mathrm{Pic}^0_{\mathcal{X}_U/U}$ maps isomorphically onto $\mu_m^{-1}$ via injection $\mathrm{d}_2^{0,1}$, so $\mathbb{Z}\cdot\mathrm{pr}_1^*\mathcal{O}(1)\oplus\mathrm{Pic}^0_{\mathcal{X}_U/U}[m](U)=\mathrm{Pic}_{\mathcal{X}_U/U}(U)$.
\item When $2\mid m$, by Lemma \ref{invariants}, the subgroup $\mathbb{Z}\cdot\mathrm{pr}_1^*\mathcal{O}(1)\oplus\mathrm{Pic}^0_{\mathcal{X}_U/U}(U)\subset\mathrm{Pic}_{\mathcal{X}_U/U}(U)$ is a proper subgroup of index $2$, hence there exists a section on some other component of the Picard scheme, say $\mathrm{Pic}^{d}_{\mathcal{X}_U/U}$. This section translates all the $m/2$ sections on $\mathrm{Pic}^0_{\mathcal{X}_U/U}(U)$ to $\mathrm{Pic}^{kd}_{\mathcal{X}_U/U}(U)$ for all $k$. Then index $2$ forces $d=m/2$ hence $\mathrm{Pic}_{\mathcal{X}_U/U}(U)=\bigsqcup_{k\in\mathbb{Z}}\mathrm{Pic}^{km/2}_{\mathcal{X}_U/U}(U).$
\end{enumerate}
Now in each of the two cases, one conclude the proposition by evaluate the degrees.
\end{proof}
Let us note that by Leray spectral sequence, the element $[L]\in\mathrm{Pic}^{m/2}_{\mathcal{X}_U/U}(U)$ is represented by a twisted line bundle $L$ on $\mathcal{X}_U$ whose associated Brauer class is $\mathrm{d}_{2}^{0,1}([L])$.
\begin{proof}[Proof of Theorem \ref{thm1}]By lemma \ref{lem-conn-components}, the theorem is converted to show that $\mathrm{H}^1(U,\mathrm{Pic}_{\mathcal{X}_U/U})=0$ and calculation of the image of $\mathrm{deg}_U\colon\mathrm{Pic}_{\mathcal{X}_U/U}\to \mathbb{Z}$. They are done in Proposition \ref{vanishing-whole-torsor} and Proposition \ref{image-deg} respectively.
\end{proof}

\section{Proof of Theorem \ref{thm2}}
We will pass to the generic point by taking colimit over open complement of curves in $U$. In order to use purity in  the excision sequence, we first remove the singularities of the curve, then remove the smooth part of the curve.
\subsection{Over the complement of curves}
Let $F\subset U$ be a closed subscheme of dimension $0$. Let $U'=U-F$. Let $i\colon Z\subset U'$ be a smooth closed subscheme. Let
$(\mathrm{pr}_1)_{U'},(\mathrm{pr}_1)_{U'-Z},(\mathrm{pr}_1)_{Z}$ be the base change of the projection $\mathrm{pr}_1\colon \mathcal{X}\to\mathbb{P}^2.$
Comparison between ramification sequences yields following commutative diagram:
\begin{equation}\xymatrix{
    \mathrm{Br}(U')[l^\infty]\ar[d]^{(\mathrm{pr}_1)_{U'}^*}\ar@{^(->}[r]&\mathrm{Br}(U'-Z)[l^\infty]\ar[d]^{(\mathrm{pr}_{1})_{U'-Z}^*}\ar[r]^-{\partial_1}&\mathrm{H}^1(Z,\mathbb{Q}_l/\mathbb{Z}_l)\ar[d]^{(\mathrm{pr}_1)_Z^*}\ar@{->}[r]^-{i_{Z,*}}&\mathrm{H}^3(U',\mathbb{Q}_l/\mathbb{Z}_l(1))\ar[d]^{(\mathrm{pr}_1)^*_{U'}}\\
    \mathrm{Br}(\mathcal{X}_{U'})[l^\infty]\ar@{^(->}[r]&\mathrm{Br}(\mathcal{X}_{U'-Z})[l^\infty]\ar[r]^-{\partial_2}&\mathrm{H}^1(\mathcal{X}_Z,\mathbb{Q}_l/\mathbb{Z}_l)\ar[r]^-{i_{\mathcal{X}_Z,*}}&\mathrm{H}^3(\mathcal{X}_{U'},\mathbb{Q}_l/\mathbb{Z}_l(1))
}\label{eqn3}\end{equation}\begin{lemma}\label{lemma-10-vanishing}
The Gysin map $i_{Z,*}\colon\mathrm{H}^1(Z,\mathbb{Q}_l/\mathbb{Z}_l)\to \mathrm{H}^3(U',\mathbb{Q}_l/\mathbb{Z}_l(1))$ is surjective.
\end{lemma}
\begin{proof}
Note that $U'-Z=\mathbb{P}^2-V(a_0a_1a_2)-\overline{Z}$ is an affine surface, by Artin vanishing
$\mathrm{H}^3(U'-Z,\mathbb{Q}_l/\mathbb{Z}_l(1))=0$. Then we conclude from the excision sequence.
\end{proof}
\begin{lemma}
The pullback $(\mathrm{pr}_1)^*_{U'}\colon\mathrm{H}^3(U',\mathbb{Q}_l/\mathbb{Z}_l(1))\to\mathrm{H}^3(\mathcal{X}_{U'},\mathbb{Q}_l/\mathbb{Z}_l(1))$ is injective.
\end{lemma}
\begin{proof}
Consider the commutative diagram induced from excision sequences
 $$\xymatrix{
\mathrm{H}^3(U,\mathbb{Q}_l/\mathbb{Z}_l(1))\ar[r]&\mathrm{H}^3(U',\mathbb{Q}_l/\mathbb{Z}_l(1))\ar[d]^{\mathrm{(pr_1)}^*_{U'}}\ar[r]^{\partial_{U'}}&\mathrm{H}^0(F,\mathbb{Q}_l/\mathbb{Z}_l(-1))\ar[d]^{(\mathrm{pr}_1)_{F}^*}\ar[r]&\mathrm{H}^4(U,\mathbb{Q}_l/\mathbb{Z}_l(1))\\
&\mathrm{H}^3(\mathcal{X}_{U'},\mathbb{Q}_l/\mathbb{Z}_l(1))\ar[r]&\mathrm{H}^0(\mathcal{X}_F,\mathbb{Q}_l/\mathbb{Z}_l(1))&
}$$
By Artin vanishing on $U$, we know that $\partial_{U'}$ is an isomorphism. 
 Since $(\mathrm{pr}_1)_F$ is surjective, we know that $(\mathrm{pr}_1)_{F}^*$ is injective. By commutativity of the diagram we know that $(\mathrm{pr}_1)_{U'}^*$ is injective. 
\end{proof}
\begin{lemma}\label{coker-h1}
The pullback $(\mathrm{pr}_1)_Z^*\colon\mathrm{H}^1(Z,\mathbb{Q}_l/\mathbb{Z}_l)\to\mathrm{H}^1(\mathcal{X}_Z,\mathbb{Q}_l/\mathbb{Z}_l)$ is injective and $$\mathrm{coker}(\mathrm{pr}_1)_Z^*\cong \mathrm{H}^0(Z,R^1(\mathrm{pr}_1)_{Z,*}\mathbb{Q}_l/\mathbb{Z}_l).$$
\end{lemma}
\begin{proof}
This follows from Leray spectral sequence along $(\mathrm{pr}_1)_Z$.
%, we know that $(\mathrm{pr}_1)_Z^*$ fits into short exact sequence $$\xymatrix{0\ar[r]&\mathrm{H}^1(Z,\mathbb{Q}_l/\mathbb{Z}_l)\ar[r]&\mathrm{H}^1(\mathcal{X}_Z,\mathbb{Q}_l/\mathbb{Z}_l)\ar[r]&\mathrm{H}^0(Z,R^1(\mathrm{pr}_1)_{Z,*}\mathbb{Q}_l/\mathbb{Z}_l)\ar[r]&0}.$$
\end{proof}
Let $\mathrm{d}_2^{0,2}$ be the differential on the $E_2$ page:
$$\mathrm{d}_2^{0,2}\colon\mathrm{H}^0(U',R^2(\mathrm{pr}_1)_{U',*}\mathbb{Q}_l/\mathbb{Z}_l(1))\to\mathrm{H}^2(U',R^1(\mathrm{pr}_1)_{U',*}\mathbb{Q}_l/\mathbb{Z}_l(1)),$$
\begin{proposition}\label{relative-brauer}
For any prime $l$, we have short exact sequence $$\xymatrix{0\ar[r]&\frac{\mathrm{Br}(\mathcal{X}_{U'-Z})[l^\infty]}{(\mathrm{pr_1})_{U'-Z}^*\mathrm{Br}(U'-Z)[l^\infty]}\ar[r]^-{\overline{\partial_2}}&\mathrm{H}^0(Z,R^1(\mathrm{pr}_1)_{Z,*}\mathbb{Q}_l/\mathbb{Z}_l)\ar[r]^-{\overline{i_{Z,*}}}&\frac{\mathrm{H}^2(U',R^1(\mathrm{pr}_{1})_{U',*}\mathbb{Q}_l/\mathbb{Z}_l(1))}{\mathrm{d}_{2}^{0,2}(\mathrm{H}^0(U',R^2(\mathrm{pr}_1)_{U',*}\mathbb{Q}_l/\mathbb{Z}_l(1)))}}$$
\end{proposition}
\begin{proof}
Let us note that Leray spectral sequence equips cohomologies with canonical decreasing filtration
$$\mathrm{F}^1\mathrm{H}^1(\mathcal{X}_Z,\mathbb{Q}_l/\mathbb{Z}_l)=\mathrm{F}^1\mathrm{H}^1\supset \mathrm{F}^0\mathrm{H}^1=\mathrm{H}^1(Z,\mathbb{Q}_l/\mathbb{Z}_l)\supset0$$
$$\mathrm{F}^3\mathrm{H}^3(\mathcal{X}_{U'},\mathbb{Q}_l/\mathbb{Z}_l(1))=\mathrm{F}^3\mathrm{H}^3\supset\mathrm{F}^2\mathrm{H}^3\supset\mathrm{F}^1\mathrm{H}^3\supset\mathrm{F}^0\mathrm{H}^3\supset 0$$
The Gysin map is compatible with  filtration 
so that $$i_{\mathcal{X}_Z,*}\mathrm{F}^k\mathrm{H}^1(\mathcal{X}_Z,\mathbb{Q}_l/\mathbb{Z}_l)\subseteq \mathrm{F}^{k+2}\mathrm{H}^3(\mathcal{X}_{U'},\mathbb{Q}_l/\mathbb{Z}_l(1)).$$ 

Together with the  previous lemmas, we rewrite diagram (\ref{eqn3}) as:
$$\xymatrix{    0\ar[r]&\frac{\mathrm{Br}(U'-Z)[l^\infty]}{\mathrm{Br}(U')[l^\infty]}\ar[d]^{\overline{(\mathrm{pr}_{1})_{U'-Z}^*}}\ar[r]&\mathrm{H}^1(Z,\mathbb{Q}_l/\mathbb{Z}_l)\ar@{^(->}[d]^{(\mathrm{pr}_1)_Z^*}\ar[r]^-{i_{Z,*}}&\mathrm{H}^3(U',\mathbb{Q}_l/\mathbb{Z}_l(1))\ar@{^(->}[d]^{(\mathrm{pr}_1)^*_{U'}}\ar[r]&0\\
    0\ar[r]&\frac{\mathrm{Br}(\mathcal{X}_{U'-Z})[l^\infty]}{\mathrm{Br}(\mathcal{X}_{U'})[l^\infty]}\ar[r]&\mathrm{H}^1(\mathcal{X}_Z,\mathbb{Q}_l/\mathbb{Z}_l)\ar[r]^-{i_{\mathcal{X}_Z,*}}&\mathrm{F}^2\mathrm{H}^3(\mathcal{X}_{U'},\mathbb{Q}_l/\mathbb{Z}_l(1))\\
}$$The proposition will follow from snake lemma once we write down the cokernels of the vertical maps.
\begin{enumerate}
\item 
By Lemma \ref{surjection}, we know that $(\mathrm{pr}_{1})_{U}^*\mathrm{Br}(U)=\mathrm{Br}(\mathcal{X}_U)$. By Lemma \ref{ramificaiton-lemma} we have $\mathrm{Br}(U)=\mathrm{Br}(U'), \ \mathrm{Br}(\mathcal{X}_{U})=\mathrm{Br}(\mathcal{X}_{U'}),$
 therefore $\mathrm{Br}(\mathcal{X}_{U'})=(\mathrm{pr}_{1,U'}^*)\mathrm{Br}({U'})\subseteq(\mathrm{pr}_{1})_{U'-Z}^*\mathrm{Br}(U'-Z)$ and 
$$\mathrm{coker}\left(\overline{(\mathrm{pr}_1)^*_{U'-Z}}\right)=
\frac{\mathrm{Br}(\mathcal{X}_{U'-Z})[l^\infty]}{(\mathrm{pr_1})_{U'-Z}^*\mathrm{Br}(U'-Z)[l^\infty]+\mathrm{Br}(\mathcal{X}_{U'})[l^\infty]}
=\frac{\mathrm{Br}(\mathcal{X}_{U'-Z})[l^\infty]}{(\mathrm{pr_1})_{U'-Z}^*\mathrm{Br}(U'-Z)[l^\infty]}.$$
\item By Lemma \ref{coker-h1}, we know that $$\mathrm{coker}(\mathrm{pr}_1)_{Z}^*\cong\mathrm{H}^0(Z,R^1(\mathrm{pr}_1)_{Z,*}\mathbb{Q}_l/\mathbb{Z}_l(1))$$
\item By convergence of first quadrant spectral sequence, we know that $\mathrm{coker}(\mathrm{pr}_1)_{U'}^*\cong\mathrm{ker}(\mathrm{d}_2^{2,1})/\mathrm{im}(\mathrm{d}_2^{0,1})$.
By excision and Artin vanishing, one has $\mathrm{H}^4(U',\mathbb{Q}_l/\mathbb{Z}_l(1))=0$ therefore $\mathrm{d}_2^{2,1}=0$, so we conclude that 
$$\mathrm{coker}(\mathrm{pr}_1)_{U'}^*\cong\frac{\mathrm{H}^2(U',R^1(\mathrm{pr}_{1})_{U',*}\mathbb{Q}_l/\mathbb{Z}_l(1))}{\mathrm{d}_{2}^{0,2}(\mathrm{H}^0(U',R^2(\mathrm{pr}_1)_{U',*}\mathbb{Q}_l/\mathbb{Z}_l(1)))}
$$
\end{enumerate}
Now we conclude the proposition by snake lemma.
\end{proof}

Let us denote $\mathbb{L}=R^1(\mathrm{pr}_1)_{U,*}\mathbb{Q}_l/\mathbb{Z}_l(1)$, so we may rewrite the result of Proposition \ref{relative-brauer} as \begin{equation}\xymatrix{0\ar[r]&\frac{\mathrm{Br}(\mathcal{X}_{U'-Z})}{(\mathrm{pr}_1)^*_{U'-Z}\mathrm{Br}(U'-Z)}[l^\infty]\ar[r]&\mathrm{H}^0(Z,\mathbb{L}|_Z(-1))\ar[r]&\frac{\mathrm{H}^2(U',\mathbb{L}|_{U'})}{\mathrm{d}_2^{0,2}(\mathrm{H}^0(U',R^2(\mathrm{pr}_1)_{U',*}\mathbb{Q}_l/\mathbb{Z}_l(1)))}}
\label{eqn4}\end{equation}
By excision sequence we know that $\mathrm{H}^2(U',\mathbb{L}|_{U'})=\mathrm{H}^2(U,\mathbb{L})$. Also note that since $U$ is irreducible, global section of locally constant sheaf does not alter when we pass to smaller open subsets, therefore the rightmost term is isomorphic to  $\frac{\mathrm{H}^2(U,\mathbb{L})}{\mathrm{d}_2^{0,2}\mathrm{H}^0(U,R^2(\mathrm{pr}_1)_{U,*}\mathbb{Q}_l/\mathbb{Z}_l(1))}$ hence independent of $F$ and $Z$ we choose.
\subsection{Passing to the generic point}
We are ready to compute the relative Brauer group of the generic Fermat curves.
\begin{proposition}\label{exact-sequence-cardinality}
Let $(\mathrm{pr}_1)_K\colon X_K\to K$ be the generic fiber of $(\mathrm{pr}_1)_U\colon\mathcal{X}_U\to U$. For any prime $l$ we have the following exact sequence,
where the direct summand is indexed by all codimension $1$ irreducible subvarieties $W\subset U$, and $\eta_W$ is the generic point of $W$.
$$\xymatrix{0\ar[r]&\frac{\mathrm{Br}(X_K)}{(\mathrm{pr}_1)_K^*\mathrm{Br}(K)}[l^\infty]\ar[r]&\bigoplus_{W\in U^{(1)}}\mathrm{H}^0(\eta_W,\mathbb{L}|_{\eta_W}(-1))\ar[r]&\frac{\mathrm{H}^2(U,\mathbb{L})}{\mathrm{d}_2^{0,2}\mathrm{H}^0(U,R^2(\mathrm{pr}_1)_{U,*}\mathbb{Q}_l/\mathbb{Z}_l(1))}}.$$
\end{proposition}
\begin{proof}For every one-dimensional reduced subscheme $W\subset U$, let us take $F=Z^{\mathrm{sing}}$, $Z=Z^{\mathrm{sm}}$, we have the exact sequence as in equation (\ref{eqn4}). 
Then we conclude by taking directed colimit over all one dimensional subschemes $W\subset U$.
\end{proof}

\begin{lemma}\label{countable}
The cohomology group $\mathrm{H}^2(U,\mathbb{L})$ has countable cardinality. 
\end{lemma}
\begin{proof}
Let $\mathbb{L}[n]\subset \mathbb{L}$ be the $n$-torsion subsheaf of $\mathbb{L}$, then $\mathbb{L}=\mathrm{colim}_n\mathbb{L}[n]$. Now note that $$\mathrm{H}^2(U,\mathbb{L})=\mathrm{colim}_{n\to\infty}\mathrm{H}^2(U,\mathbb{L}[n]),$$ it reduces to the well-known result that \'etale cohomology of finite locally constant system on quasi-projective varieties are finite.
\end{proof}
The following elementary lemma says that every generically finite covering of complex algebraic manifolds splits along uncountably many divisors in the base.
\begin{lemma}\label{split}
Let $f\colon X\to Y$ be a generically finite morphism of complex algebraic varieties, let $Y^{\mathrm{sm}}\subset Y$ be the smooth loci of $Y$, let $Y^\circ\subset Y^{\mathrm{sm}}$ be the locus over which $f$ is finite \'etale, and let $X^\circ=f^{-1}(Y^\circ)$.
Then for any $x\in X^\circ$, there exists a hypersurface $H\subset X$ passing through $x$ and maps birationally into $Y$.
\end{lemma}
\begin{proof}
Taking compactificaiton and resolution of singularities, we may assume $X,Y$ are smooth projective varieties.
Let $H\subset X$ be any hypersurface that passes through $x$ and avoids $f^{-1}(f(x))\backslash\{x\}$, then $H$ satisfies desired condition because
\begin{enumerate}
\item By upper semi-continuity of fiber dimensions, we know that $f|_H$ is generically finite,
\item If $\mathrm{deg}(f|_H)\geq2$, then at least two points in $f^{-1}(x)$ lies in $H$, contradiction with $f^{-1}(f(x))\cap H=x$. 
\end{enumerate}Such $H$ can be taken in any very ample linear system on $X$.
\end{proof}

\begin{proposition}\label{uncountable}
The group $\bigoplus_{W\in U^{(1)}}\mathrm{H}^0(\eta_W,\mathbb{L}|_{\eta_W}(-1))$ is uncountable.
\end{proposition}
\begin{proof}
Let us consider the finite cover $f\colon \widetilde{U}=U\to U,[a_0:a_1:a_2]\mapsto [a_0^m:a_1^m:a_2^m]$. If the cardinality is countable, then necessarily $I=\{W\in U^{(1)}|\mathrm{H}^0(\eta_W,\mathbb{L}|_{\eta_W}(-1))\neq0\}$ is a countable set.
However, by Lemma \ref{split}, for any point $u\in \widetilde{U}(\mathbb{C})\backslash \bigcup_{W\in I} W(\mathbb{C})$, there exist a curve $H\subset U$ passing through $u$ such that $f|_{H}\colon H_u\to f(H)$ is birational. Then for $W'=f(H)$, we have $\mathrm{H}^0(\eta_{W'},\mathbb{L}|_{\eta_{W'}}(-1))=\mathrm{H}^0(\eta_H,f^*\mathbb{L}|_{\eta_H})\neq0$ because $f^*\mathbb{L}$ is trivial local system on $\widetilde{U}$ with stalk $(\mathbb{Q}_l/\mathbb{Z}_l)^{2g}$. 
\end{proof}

\begin{proof}[Proof of Theorem \ref{thm2}] 
By Proposition \ref{exact-sequence-cardinality}, Proposition \ref{uncountable} and Lemma \ref{countable}, we know that the $l$-primary part of relative Brauer group is calculated as the kernel of a map from an uncountable abelian group to a countable abelian group. Hence for all primes $l$, we have $\mathrm{card}(\mathrm{Br}(X_K)/(\mathrm{pr}_1)_K^*\mathrm{Br}(K))[l^\infty]>\aleph_0$. By Leray spectral sequence and Artin vanishing one has exact sequence $\xymatrix{\mathrm{Br}(K)\ar[r]&\mathrm{Br}(X_K)\ar[r]&\mathrm{H}^1(K,\mathrm{Pic}_{X/K})\ar[r]&0}$
hence $\mathrm{H}^1(K,\mathrm{Pic}_{X_K/K})$ is uncountable. Now Lemma \ref{lem-conn-components} shows that $\mathrm{H}^1(K,\mathrm{Pic}^0_{X/K})$ is a finite extension of $\mathrm{H}^1(U,\mathrm{Pic}_{X_K/K})$, hence also uncountable.
\end{proof}

\section{Partial results on \texorpdfstring{$\mathcal{M}_g$}{Mg}} In this section, we give some parallel results of Theorem \ref{thm1} and Theorem \ref{thm2} for the universal family $\pi\colon\mathcal{M}_{g,1}\to\mathcal{M}_g$ of smooth genus $g$ curves. Throughout this section we will assume $g\geq3$, so that $\mathcal{M}_g$ is generically automorphism free.
We use $K$ to denote the function field of $\mathcal{M}_g$, and denote generic fiber of $\pi\colon\mathcal{M}_{g,1}\to\mathcal{M}_g$ by $C/K$. We aim to show that
\begin{enumerate}
\item All torsors for the universal Jacobian over $\mathcal{M}_g$ are canonical.
\item Over the generic point, there are uncountably many non-canonical torsors.
\end{enumerate}
While we are not able to prove these, we reduce problem (1) to some explicit calculation of cohomology of mapping class groups, and we show that the universal Jacobian of the generic marked curve are necessarily trivial.

\subsection{Torsors of the universal Jacobian}
\begin{lemma}\label{lemma-M_g-connected}
We have short exact sequence
$$\xymatrix{0\ar[r]&\mathbb{Z}/(2g-2)\mathbb{Z}\ar[r]^-{\delta}&\mathrm{H}^1(\mathcal{M}_g,\mathrm{Pic}^0_{\mathcal{M}_{g,1}/\mathcal{M}_g})\ar[r]&\mathrm{H}^1(\mathcal{M}_g,\mathrm{Pic}_{\mathcal{M}_{g,1}/\mathcal{M}_g})\ar[r]&0},$$
where $\delta$ sends $\overline{d}\in\mathbb{Z}/(2g-2)\mathbb{Z}$ to the class of $[\mathrm{Pic}^d_{\mathcal{M}_{g,1}/\mathcal{M}_g}].$
\end{lemma}
\begin{proof}
As in the proof of Lemma \ref{lem-conn-components}, it suffices to show that cokernel of the degree map $$\mathrm{deg}_U\colon\mathrm{Pic}_{\mathcal{M}_{g,1}/\mathcal{M}_g}(U)\to\mathbb{Z}$$ is $\mathbb{Z}/(2g-2)\mathbb{Z}$. 
Since $U$ is smooth, every rational section of an abelian scheme over $U$ necessarily extends to a regular section, so
$\mathrm{Pic}_{\mathcal{M}_{g,1}/\mathcal{M}_g}(U)=\mathrm{Pic}_{C/K}(K)$, then we conclude from the strong Franchetta theorem \cite[Theorem 5.1]{Schroeer-Strong-Franchetta}. 
\end{proof}

\begin{lemma}\label{Pirisi-lemma} If $g\geq4$, then 
$\mathrm{Br}(\mathcal{M}_{g,n})=0$ for all $n\geq0$. \end{lemma}
\begin{proof}
See
{\cite[Theorem 4.1]{Pirisi-Brauer-Group-Universal}} for a proof.
The proof goes by comparing the Brauer group with analytic Brauer group, so the argument works over $\mathbb{C}$. When $g=3$, there is an extra element $\mathrm{Br}(\mathcal{M}_{3})\cong\mathbb{Z}/2\mathbb{Z}$, see \cite[Corollary 31]{Pirisi-Brauer-genus-3}. General calculation of $\mathrm{Br}(\mathcal{M}_{3,n})$ is not yet known.
\end{proof} 
Now we reduce the problem of classifiying torsors of universal Jacobian to a purely group theoretic question.
Let $\mathrm{Mod}(g)$ and $\mathrm{Mod}(g,1)$ be the fundamental group of $\mathcal{M}_g$ and $\mathcal{M}_{g,1}$, which are also the mapping class group of a smooth projective genus $g$ curve, and a marked genus $g$ curve. The projection $\pi\colon\mathcal{M}_{g,1}\to\mathcal{M}_g$ induces a map of the fundamental groups $p\colon\mathrm{Mod}(g,1)\to\mathrm{Mog}(g)$.

\begin{proposition}\label{part-res-1}
A necessary and sufficient condition for $$\mathrm{H}^1_{\mathrm{et}}(\mathcal{M}_g,\mathrm{Pic}^0_{\mathcal{M}_{g,1}/\mathcal{M}_g})=\langle[\mathrm{Pic}^1_{\mathcal{M}_{g,1}/\mathcal{M}_g}]\rangle\cong\mathbb{Z}/(2g-2)\mathbb{Z}$$ is that: for all $n$, the pullback map induces injection on third cohomology
$$p^*\colon \mathrm{H}^3(\mathrm{Mod}(g),\mathbb{Z}/n\mathbb{Z})\to\mathrm{H}^3(\mathrm{Mod}(g,1),\mathbb{Z}/n\mathbb{Z}).$$

\end{proposition}
\begin{proof}
The Leray spectral sequence for $\mathbb{G}_m$ along $\pi$ gives the following short exact sequence 
$$\xymatrix{0\ar[r]&\frac{\mathrm{Br}(\mathcal{M}_{g,1})}{\pi^*\mathrm{Br}(\mathcal{M}_{g})}\ar[r]&\mathrm{H}_{\mathrm{et}}^1(\mathcal{M}_g,\mathrm{Pic}_{\mathcal{M}_{g,1}/\mathcal{M}_g})\ar[r]&\mathrm{H}_{\mathrm{et}}^3(\mathcal{M}_{g},\mathbb{G}_m)\ar[r]^-{\pi^*}&\mathrm{H}_{\mathrm{et}}^3(\mathcal{M}_{g,1},\mathbb{G}_m)}$$
By Lemma \ref{lemma-M_g-connected} and Lemma \ref{Pirisi-lemma}, the calculation $\mathrm{H}^1_{\mathrm{et}}(\mathcal{M}_g,\mathrm{Pic}^0_{\mathcal{M}_{g,1}/\mathcal{M}_g})\cong\mathbb{Z}/(2g-2)\mathbb{Z}$ is equivalent to injectivity of $\pi^*\colon\mathrm{H}_{\mathrm{et}}^3(\mathcal{M}_g,\mathbb{G}_m)\to\mathrm{H}_{\mathrm{et}}^3(\mathcal{M}_{g,1},\mathbb{G}_m)$. 
Note that $\mathrm{H}_{\mathrm{et}}^3(\mathcal{M}_g,\mathbb{G}_m)$ is torsion, it suffices to show that for all $n\geq1$, $\mathrm{H}_{\mathrm{et}}^3(\mathcal{M}_{g},\mathbb{G}_m)[n]\to\mathrm{H}_{\mathrm{et}}^3(\mathcal{M}_{g,1},\mathbb{G}_m)[n]$ is injective. By Lemma \ref{Pirisi-lemma} and Kummer sequence, it is equivalent to show injectivity of the pullback map $$\pi^*\colon\mathrm{H}_{\mathrm{et}}^3(\mathcal{M}_g,\mu_n)\to\mathrm{H}_{\mathrm{et}}^3(\mathcal{M}_{g,1},\mu_n).$$
Note that \'etale cohomology is the same as complex cohomology  for finite torsion sheaves on quasi-projective varieties, by
the Hochschild-Serre spectral sequence   
for the Teichm\"uller uniformization
$$\mathrm{H}^p(\Gamma,\mathrm{H}^q(\mathbb{C}^{3g-3},\mu_n))\Rightarrow\mathrm{H}^p(\mathcal{M}_g,\mu_n)$$ we have $\mathrm{H}_{\mathrm{et}}^3(\mathcal{M}_g,\mu_n)\cong\mathrm{H}^3(\mathrm{Mod}(g),\mathbb{Z}/n\mathbb{Z})$, and similarly $\mathrm{H}_{\mathrm{et}}^3(\mathcal{M}_{g,1},\mu_n)\cong\mathrm{H}^3(\mathrm{Mod}(g,1),\mathbb{Z}/n\mathbb{Z})$, by naturality of spectral sequence, the problem is equivalent to show injectivity of $p^*\colon\mathrm{H}^3(\mathrm{Mod}(g),\mu_n)\to\mathrm{H}^3(\mathrm{Mod}(g,1),\mu_n)$. \end{proof}
The question can also be posed for the universal family of $n$-marked curves $\pi_n\colon\mathcal{M}_{g,n+1}\to\mathcal{M}_{g,n}$. We show that the universal Jacobian of the generic marked genus $g$ curves do not admit non-trivial torsors:
\begin{proposition}\label{marked}For all  $g\geq4$, $n\geq1$, we have $\mathrm{H}^1(\mathcal{M}_{g,n},\mathrm{Pic}^0_{\mathcal{M}_{g,n+1}/\mathcal{M}_{g,n}})=0.$
\end{proposition}
\begin{proof}Note that $\pi_n\colon \mathcal{M}_{g,n+1}\to\mathcal{M}_{g,n}$ admit sections $\delta_i$ given by the $i$-th marked point, so the pullback map $\pi_n^*\colon\mathrm{H}^3(\mathcal{M}_{g,n},\mathbb{G}_m)\to\mathrm{H}^3(\mathcal{M}_{g,n+1},\mathbb{G}_m)$ splits as a direct summand via $\delta_1^*$ thus naturally injective. Since the torsors $\mathrm{Pic}^d_{\mathcal{M}_{g,n+1}/\mathcal{M}_{g,n}}$ are trivial, by Lemma \ref{lemma-M_g-connected} we have $\mathrm{H}^1(\mathcal{M}_{g,n},\mathrm{Pic}^0_{\mathcal{M}_{g,n+1}/\mathcal{M}_{g,n}})=\mathrm{H}^1(\mathcal{M}_{g,n},\mathrm{Pic}_{\mathcal{M}_{g,n+1}/\mathcal{M}_{g,n}})$
Then we use the same argument as in Proposition \ref{marked}.
\end{proof}

\subsection{Over the generic point}
We wish to show that for $g\geq3$, the Jacobian of the generic genus $g$ curve has uncountably many non-isomorphic torsors. The proof does not follow from Theorem \ref{thm2}, the main difference is that we do not have an analogue of Lemma \ref{lemma-10-vanishing}: we do not know if $\mathrm{H}^3(\mathcal{M}_g\backslash\overline{Z},\mu_{l}^r)=0$ holds for a cofinal system of divisors $\overline{Z}\subset\mathcal{M}_g$. Other non-vanishing terms only contribute a countable part to the final result.

Let $F\subset \mathcal{M}_{g}$ be a closed substack of codimension at least $2$. Let $Z\subset \mathcal{M}_g-F$ be a close embedded smooth substack of codimension $1$, comparing the ramification sequences, we get the following commutative diagram:
$$\xymatrix{
    \frac{\mathrm{Br}(\mathcal{M}_g\backslash\overline{Z})}{\pi^*\mathrm{Br}(\mathcal{M}_g\backslash F)}[l^\infty]\ar[d]^{p_2}\ar@{^(->}[r]&\mathrm{H}^1(Z,\mathbb{Q}_l/\mathbb{Z}_l)\ar[d]^{p_3}\ar[r]^-{i_{Z,*}}&\mathrm{H}^3(\mathcal{M}_g\backslash F,\mathbb{Q}_l/\mathbb{Z}_l(1))\ar[d]^{p_4}\ar[r]&\mathrm{H}^3(\mathcal{M}_g\backslash\overline{Z},\mathbb{Q}_l/\mathbb{Z}_l(1))\ar[d]\\
\frac{\mathrm{Br}(\mathcal{M}_{g,1}\backslash\pi^{-1}\overline{Z})}{\mathrm{Br}(\mathcal{M}_{g,1}\backslash\pi^{-1}F)}[l^\infty]\ar@{^(->}[r]&\mathrm{H}^1(\pi^{-1}Z,\mathbb{Q}_l/\mathbb{Z}_l)\ar[r]&\mathrm{H}^3(\mathcal{M}_{g,1}\backslash\pi^{-1}F,\mathbb{Q}_l/\mathbb{Z}_l(1))\ar[r]^-{r_Z}&\mathrm{H}^3(\mathcal{M}_{g,1}\backslash\pi^{-1}\overline{Z},\mathbb{Q}_l/\mathbb{Z}_l(1))
.}$$

\begin{proposition}\label{part-res-2}Let $C/K$ be the generic fiber of $\pi\colon\mathcal{M}_{g,1}\to\mathcal{M}_g$, then
$$\mathrm{card}\left(\mathrm{colim}_{F,Z}\left(\frac{\mathrm{Ker}(r_Z)}{\mathrm{Im}(p_4\circ i_{Z,*})}\right)\right)\leq\aleph_0\Rightarrow\mathrm{card}\left(\mathrm{Br}(C)/\mathrm{Br}(K)\right)>\aleph_0.$$
\end{proposition}
\begin{proof}As in the proof of Theorem \ref{thm2}, applying five lemma to the comparison diagram, we get exact sequence:
$$\xymatrix{0\ar[r]&
\mathrm{Im}(i_{Z,*})\cap\mathrm{Ker}(p_4)\ar[r]&
\frac{\mathrm{Br}(\mathcal{M}_{g,1}\backslash\pi^{-1}Z)[l^\infty]}{p_2^*\mathrm{Br}(\mathcal{M}_g\backslash\overline{Z})+\mathrm{Br}(\mathcal{M}_{g,1}\backslash\pi^{-1}F)}\ar[r]&\mathrm{H}^0(Z,R^1\pi_*\mathbb{Q}_l/\mathbb{Z}_l|_Z)\ar[r]&\frac{\mathrm{Ker}(r_Z)}{\mathrm{Im}(p_4\circ i_{Z,*})}\ar[r]&0.}$$Let us note that
\begin{enumerate}
\item The group $\mathrm{Im}(i_{Z,*})\cap\mathrm{Ker}(p_4)$ is isomorphic to a subgroup of $\mathrm{H}^3(\mathcal{M}_g,\mathbb{Q}_l/\mathbb{Z}_l(1))$, hence $$\mathrm{card}\left(\mathrm{colim}_{F,Z}\left(\mathrm{Im}(i_{Z,*})\cap\mathrm{Ker}(p_4)\right)\right)\leq\aleph_0.$$
Reason: First let us  assume that $F$ is smooth. Consider the comparison diagram of excision sequences
$$\xymatrix{
\mathrm{H}^3(\mathcal{M}_g,\mathbb{Q}_l/\mathbb{Z}_l(1))\ar@{^(->}[r]\ar[d]&\mathrm{H}^3(\mathcal{M}_g\backslash F,\mathbb{Q}_l/\mathbb{Z}_l(1))\ar[d]^{p_4}\ar[r]&\mathrm{H}^0(F,\mathbb{Q}_l/\mathbb{Z}_l(-1))\ar[d]^{\cong}&\\
\mathrm{H}^3(\mathcal{M}_{g,1},\mathbb{Q}_l/\mathbb{Z}_l(1))\ar@{^(->}[r]&\mathrm{H}^3(\mathcal{M}_{g,1}\backslash\pi^{-1}F,\mathbb{Q}_l/\mathbb{Z}_l(1))\ar[r]&\mathrm{H}^0(\pi^{-1}F,\mathbb{Q}_l/\mathbb{Z}_l(-1))
}$$
Diagram chasing tells us $p_4^{-1}(\mathrm{H}^3(\mathcal{M}_{g,1},\mathbb{Q}_l/\mathbb{Z}_l))=\mathrm{H}^3(\mathcal{M}_g,\mathbb{Q}_l/\mathbb{Z}_l(1)),$ therefore we have $$\mathrm{Im}(i_{Z,*})\cap\mathrm{Ker}(p_4)\subseteq\mathrm{Ker}(p_4)=p_4^{-1}(0)\subset
\mathrm{H}^3(\mathcal{M}_g,\mathbb{Q}_l/\mathbb{Z}_l(1))
.$$
Now for general $F$, we take a stratication $F\supset F^1\supset\cdots\supset F^m$ of singularities of $F$, and iteratedly replace $(\mathcal{M}_g,F)$ by $(\mathcal{M}_g-F_i,F-F_i)$ for $i=m,\cdots,1$. By excision sequence and purity, removing high codimension smooth loci does not change cohomologies, so argument works through.
\item For any very general point $x\in\mathcal{M}_g$, there exists a divisor $Z_x$ passing through $x$ such that $R^1\pi_*\mathbb{Q}_l/\mathbb{Z}_l(1)$ admit sections on $Z_x$.
Let us consider the finite covering 
$\mathcal{M}_{g}[l]\to\mathcal{M}_g$
by moduli of genus $g$ curves with level-$l$ structure, then $R^1\pi_{*}\mathbb{Q}_l/\mathbb{Z}_l(1)$ admits a trivial subsheaf $\underline{(\mathbb{Z}/l\mathbb{Z})^{2g}}$. By Lemma \ref{split}, we may take divisors $H\subset\mathcal{M}_{g}[l]$ that maps birationally into $\mathcal{M}_g$. Let $Z=H$ then $\mathrm{H}^0(k(Z),R^1\pi_*\mathbb{Q}_l/\mathbb{Z}_l(1))\supseteq(\mathbb{Z}/l\mathbb{Z})^{2g}$. As $Z$ runs through all the uncountably height one points in $\mathcal{M}_g$ over which the finite cover $\mathcal{M}_{g}[l]\to\mathcal{M}_g$ splits, the colimit $\bigoplus_{Z\in\mathcal{W}^{(1)}}\mathrm{H}^0(\eta_Z,R^1\pi_*\mathbb{Q}_l/\mathbb{Z}_l(1))$ is uncountable.
\end{enumerate}
Now we conclude by taking directed colimit of the exact sequence over all divisors $W\subset\mathcal{M}_g$, where we take $F=W^{\mathrm{sing}}$ and $Z=W\backslash F$. %Note that by iteratedly apply the purity lemma, we have $\mathrm{Br}(\mathcal{M}_{g,1}\backslash F)=\mathrm{Br}(\mathcal{M}_{g,1})$ since $\mathcal{M}_{g,1}$ is smooth and $F$ has codimension at least two.
\end{proof}

\bibliographystyle{alpha}
\bibliography{references}

\begin{thebibliography}{DPSW16}

\bibitem[Bea22]{zbMATH07495474}
Arnaud Beauville.
\newblock A remark on the generalized {Franchetta} conjecture for {{\(K3\)}} surfaces.
\newblock {\em Math. Z.}, 300(4):3337--3340, 2022.

\bibitem[CTS21]{CT-Brauer-Grothendieck-Group}
Jean-Louis Colliot-Th{\'e}l{\`e}ne and Alexei~N. Skorobogatov.
\newblock {\em The {Brauer}-{Grothendieck} group}, volume~71 of {\em Ergeb. Math. Grenzgeb., 3. Folge}.
\newblock Cham: Springer, 2021.

\bibitem[CTS23]{colliotthélène2023lowdegreeunramifiedcohomology}
Jean-Louis Colliot-Thélène and Alexei~N. Skorobogatov.
\newblock Low degree unramified cohomology of generic diagonal hypersurfaces, 2023.

\bibitem[DPSW16]{Galois-action-homology-MR3596577}
Rachel Davis, Rachel Pries, Vesna Stojanoska, and Kirsten Wickelgren.
\newblock Galois action on the homology of {F}ermat curves.
\newblock In {\em Directions in number theory}, volume~3 of {\em Assoc. Women Math. Ser.}, pages 57--86. Springer, [Cham], 2016.

\bibitem[FL23]{zbMATH07690544}
Lie Fu and Robert Laterveer.
\newblock Special cubic four-folds, {{\(K3\)}} surfaces, and the {Franchetta} property.
\newblock {\em Int. Math. Res. Not.}, 2023(10):8872--8902, 2023.

\bibitem[FP21]{Pirisi-Brauer-Group-Universal}
Roberto Fringuelli and Roberto Pirisi.
\newblock The {Brauer} group of the universal moduli space of vector bundles over smooth curves.
\newblock {\em Int. Math. Res. Not.}, 2021(18):13609--13644, 2021.

\bibitem[Fra54]{zbMATH03091363}
A.~Franchetta.
\newblock Sulle serie lineari razionalmente determinate sulla curva a moduli generali di dato genere.
\newblock {\em Matematiche}, 9:126--147, 1954.

\bibitem[LP24]{Pirisi-Brauer-genus-3}
Andrea~Di Lorenzo and Roberto Pirisi.
\newblock The brauer groups of moduli of genus three curves and plane curves, 2024.

\bibitem[LZ22]{zbMATH07577489}
Zhiyuan Li and Xun Zhang.
\newblock Deligne-{Beilinson} cohomology of the universal {{\(K3\)}} surface.
\newblock {\em Forum Math. Sigma}, 10:28, 2022.
\newblock Id/No e64.

\bibitem[O'G13]{zbMATH06448861}
Kieran~G. O'Grady.
\newblock Moduli of sheaves and the {Chow} group of {{\(K3\)}} surfaces.
\newblock {\em J. Math. Pures Appl. (9)}, 100(5):701--718, 2013.

\bibitem[PSY17]{zbMATH07004433}
Nebojsa Pavic, Junliang Shen, and Qizheng Yin.
\newblock On {O}'{Grady}'s generalized {Franchetta} conjecture.
\newblock {\em Int. Math. Res. Not.}, 2017(16):4971--4983, 2017.

\bibitem[Sch03]{Schroeer-Strong-Franchetta}
Stefan Schr{\"o}er.
\newblock The strong {Franchetta} conjecture in arbitrary characteristics.
\newblock {\em Int. J. Math.}, 14(4):371--396, 2003.

\end{thebibliography}
\end{document}